\documentclass{article}

\RequirePackage[OT1]{fontenc}
\RequirePackage[%lnms,
amsthm,amsmath%,natbib
]{imsart}

\usepackage{graphicx}
\usepackage{latexsym,amsmath}
\usepackage{amsmath,amsthm}
\usepackage{amsfonts}
\usepackage[psamsfonts]{amssymb}
\usepackage{enumerate}
\usepackage{url}

\startlocaldefs
\theoremstyle{plain} 
\newtheorem{theorem}{Theorem}%[section]
\newtheorem{corollary}%[theorem]
{Corollary}

{Proposition}

\theoremstyle{definition} 
{Definition}
\theoremstyle{definition} 
\newtheorem{ex}{Example}
\theoremstyle{remark} 

\theoremstyle{definition}
\newtheorem{remark}%[theorem]
{Remark}
\newtheorem*{remark*}{Remark}
\newcommand{\beqa}{\begin{eqnarray}}
\newcommand{\eeqa}{\end{eqnarray}}
 
\newcommand{\bseq}{\begin{subequations}}
 
\newcommand{\eseq}{\end{subequations}}

\newcommand{\N}{\mathbb N}%{\{1,2,\dots\}}
\newcommand{\dd}{\partial}

\renewcommand{\dd}{{\operatorname{d}}}

\newcommand{\ffrown}{\text{\raisebox{3pt}[0pt][0pt]{$\frown$}}}

\renewcommand{\O}{\underset{\ffrown}{<}}
\newcommand{\OG}{\underset{\ffrown}{>}}
\newcommand{\OO}{\mathrel{\text{\raisebox{2pt}{$\O$}}}}
\newcommand{\OOG}{\mathrel{\text{\raisebox{2pt}{$\OG$}}}}

\newcommand{\Ga}{\Gamma}
\newcommand{\si}{\sigma}
\newcommand{\la}{\lambda}

\renewcommand{\Psi}{\overline{\Phi}}

\newcommand{\fl}[1]{\lfloor#1\rfloor}
\newcommand{\ce}[1]{\lceil#1\rceil}

\renewcommand{\Re}{\operatorname{\mathsf{Re}}}
\renewcommand{\Re}{\operatorname{\mathfrak{Re}}}
\renewcommand{\Im}{\operatorname{\mathfrak{Im}}}

\newcommand{\Pin}{\operatorname{Pin}}

\newcommand{\ii}[1]{\,\mathbf{I}\{#1\}}

\newcommand{\PP}{\operatorname{\mathsf{P}}} 
\newcommand{\E}{\operatorname{\mathsf{E}}}

\newcommand{\R}{\mathbb{R}}
\newcommand{\C}{\mathbb{C}}

\newcommand{\vp}{\varepsilon}

\newcommand{\tPi}{{\tilde{\Pi}}}

\renewcommand{\le}{\leqslant}
\renewcommand{\ge}{\geqslant}

\endlocaldefs

\begin{document}

\begin{frontmatter}

\title{Positive-part moments\\
via the Fourier-Laplace transform}
\runtitle{%Positive moments via ch.\ f.
Positive-part moments
via the Fourier-Laplace transform}
%\date{\today}

\begin{aug}
%\author{\fnms{Iosif} \snm{Pinelis}\ead[label=e1]{ipinelis@math.mtu.edu}}
%\runauthor{Iosif Pinelis}

\author{\fnms{Iosif} \snm{Pinelis}\thanksref{t1}\thanksref{t2}\ead[label=e1]{ipinelis@mtu.edu}}
  \thankstext{t1}{Department of Mathematical Sciences, 
Michigan Technological University, 
Houghton, Michigan, USA. 
%E-mail: \texttt{ipinelis@math.mtu.edu}
}
  \thankstext{t2}{Supported by NSF grant DMS-0805946}
\runauthor{Iosif Pinelis}

%\affiliation{Michigan Technological University}

%\address{Department of Mathematical Sciences\\
%Michigan Technological University\\
%Houghton, Michigan 49931, USA\\
%E-mail: \printead[ipinelis@math.mtu.edu]{e1}}
\end{aug}

\begin{abstract}
{%\normalsize 
%Under general moment conditions, i
Integral expressions for positive-part moments $\E X_+^p$ ($p>0$) of random variables $X$ are presented, in terms of the Fourier-Laplace or Fourier transforms of the distribution of $X$. A necessary and sufficient condition for the validity of such an expression is given. This study was motivated by extremal problems in probability and statistics, where one needs to evaluate such positive-part moments. 
} 
\end{abstract}

%60E10 Characteristic functions; other transforms
%42A38 Fourier and Fourier-Stieltjes transforms and other transforms of Fourier type
%
%60E07 Infinitely divisible distributions; stable distributions
%60E15 Inequalities; stochastic orderings 
%# 42A55 Lacunary series of trigonometric and other functions; Riesz products
%# 42A61 Probabilistic methods



\end{frontmatter}

\begin{center}
{\footnotesize Address:

Department of Mathematical Sciences\\
Michigan Technological University\\
Houghton, Michigan 49931, USA\\
E-mail: \texttt{ipinelis@math.mtu.edu}\\

Telephone: 906-487-2108
}
\end{center}

\begin{center}
{\footnotesize
Running head:

Positive-part moments via the Fourier-Laplace transform
}
\end{center}

\theoremstyle{plain} 
%\newtheorem{theorem}{Theorem}[section]
%\newtheorem{corollary}[theorem]{Corollary}
%\newtheorem*{main}{Main~Theorem}
%\newtheorem{lemma}{Lemma}[subsection]
%\newtheorem{proposition}[theorem]{Proposition}
%\newtheorem{conjecture}{Conjecture}
%\theoremstyle{definition} 
%\newtheorem{definition}[theorem]{Definition}
%\theoremstyle{definition} 
%\newtheorem{ex}{Example}
%\theoremstyle{remark} 
%\newtheorem{exer}{Exercise}
%\theoremstyle{remark} 
%\newtheorem{remark}[theorem]{Remark}
%\newtheorem*{remark*}{Remark}
%\numberwithin{equation}{section}

\newpage

\section{Introduction}\label{intro}
%For any $x\in\R$, let $x_+:=0\vee x=\max(0,x)$, the positive part of $x$, and $x_+^p:=(x_+)^p$, for any $p>0$. 
In a number of extremal problems in probability and statistics (see e.g.\ \cite{eaton1,eaton2,pin94,pin98,pin99,pin-eaton,bent-liet02,bent-jtp,bent-ap,binom,normal,asymm,be-64pp,%pin-houdre,
08hoeff,pin08}) one needs to evaluate the positive-part moments $\E X_+^p$ of random variables (r.v.'s) $X$, where $x_+:=0\vee x=\max(0,x)$, the positive part of $x$, and $x_+^p:=(x_+)^p$, for any $x\in\R$ and $p>0$. 
%
%The problem of evaluating the positive-part moments $\E X_+^p$ of random variables (r.v.'s) $X$ arises, in particular, in a number of extremal problems in probability and statistics; see e.g.\ \cite{eaton1,eaton2,pin94,pin98,pin99,pin-eaton,bent-liet02,bent-jtp,bent-ap,binom,normal,asymm,be-64pp,%pin-houdre,
%08hoeff,pin08}. 

In particular, an effective procedure was needed in \cite{pin08} to compute $\E X_+^p$ for $p=3$ and r.v.'s of the form $X=\Ga_{\mu,a^2}+y\Pi_{b^2/y^2}$, where $a,b,y$ are positive real numbers, $\mu\in\R$, $\Ga_{\mu,a^2}$ and $\Pi_{b^2/y^2}$ are independent r.v.'s, 
$\Ga_{\mu,a^2}$ has the normal distribution with parameters $\mu$ and $a^2$, and $\Pi_{b^2/y^2}$
has the Poisson distribution with parameter $b^2/y^2$. For purely normal r.v.'s $X$ (without the Poisson component) such a  computation is easy. 
However, the naive approach, by the formula $\E(\Ga_{\mu,a^2}+y\Pi_{b^2/y^2})_+^p=\sum_{j=0}^\infty\E(\Ga_{\mu,a^2}+yj)_+^p%\times%\break
\PP(\Pi_{b^2/y^2}=j)$ did not work well, especially when $y$ is small. 
On the other hand, the series of the form $\sum_{j=0}^\infty e^{jz}\PP(\Pi_{b^2/y^2}=j)$ (which is the Laplace-Fourier transform of the Poisson distribution) is easily computable. Therefore, a natural idea was to perform a harmonic analysis of the function $\R\ni x\mapsto f_p(x):=x_+^p$. 
This can be easily done by finding its Laplace-Fourier transform. % --- see the proof of Theorem~\ref{lem:ga} later in this paper. 
Then, once the function $f_p$ is decomposed into harmonics (i.e., exponential functions $x\mapsto e^{zx}$ for $z\in\C$), one almost immediately obtains a general expression for the positive-part moments of a r.v.\ $X$ in terms of the Laplace-Fourier transform of the distribution of $X$, as provided indeed by Theorem~\ref{lem:ga} in this paper.    
Such expressions turn out to be computationally effective, as well as the ones in terms of the characteristic function (c.f.) $\R\ni t\mapsto\E e^{itX}$ that are obtained as corollaries. 
%We also show that the moment assumption on the left tail of the distribution of $X$ can be relaxed (see 
Theorem~\ref{cor:improp} of this paper %albeit for for the price that the corresponding integral be understood as improper (at $0$). In fact, 
provides a necessary and sufficient condition for such a representation.  

A subsequent literature search has revealed only one paper, Brown \cite{brown70}, that contains formulas for $\E X_+^p$ in terms of the c.f.\ of $X$. However, the constant-sign argument used in \cite{brown70} does not actually seem to work for $p\in(0,2)$. 
The results in \cite{brown70} are based on the method developed by von Bahr \cite{vonBahr} to express the absolute moments $\E|X|^p$ in terms of the c.f.\ of $X$. Other papers containing such expressions for absolute moments include \cite{hsu,%pitman,
brown70,brown72}; see also \cite[\S\S1.8.6--1.8.8]{petrov95} and \cite[\S11.4]{kawata}. 
%Some heuristics is also offered in \cite{lukacs} for formula (2.3.11) therein for the absolute moments of odd orders, but the formula does not work even for the standard normal distribution. 

As was explained, our approach differs from the previous ones. We begin with expressions of positive-part moments in terms in the Fourier-Laplace transform $z\mapsto\E e^{zX}$ with $\Re z\ne0$, and then use the Cauchy integral theorem to express $\E X_+^p$ in terms of the c.f.\ of $X$. 
It should be clear that expressions for the absolute moments will follow trivially from expressions for the positive-part moments. 

\section{Results}\label{results}
Let $X$ be any r.v. Let $s_1$ and $s_2$ be any real numbers such that $s_1\le0\le s_2$ and $\E e^{sX}<\infty$ for all $s\in[s_1,s_2]$.  
Let $p$ be any positive real number, and then let 
\begin{equation*}%\label{eq:ell}
k:=k(p):=\fl p\quad\text{and}\quad \ell:=\ell(p):=\ce{p-1},
\end{equation*}
respectively 
the integer part of $p$ and the smallest integer that is no less than $p-1$, so that $k\le p<k+1$, $\ell<p\le\ell+1$, and $\ell\le k$. 
%, and $m:=m(p):=\fl{\frac k2}$, the integer part of $p$. 
For all complex $z$ and all $m=-1,0,1,\dots$, let 
\begin{align*}
	e_m(z):=e^z-\sum%{\scriptstyle{\sum}}
	_{j=0}^m\frac{z^j}{j!},\quad 
	c_{2m}(z)&:=(-1)^{m+1}\Big(\cos z-\sum%{\scriptstyle{\sum}}
	_{j=0}^m(-1)^j\frac{z^{2j}}{(2j)!}\Big),\\
	s_{2m+1}(z)&:=(-1)^{m+1}\Big(\sin z-\sum%{\scriptstyle{\sum}}
	_{j=0}^m(-1)^j\frac{z^{2j+1}}{(2j+1)!}\Big),
\end{align*}
with the convention $\sum_{j=0}^{-1}a_j:=0$ for any $a_j$'s, so that $e_{-1}(z)\equiv e^z$, $c_{-2}(z)\equiv \cos z$, and $s_{-1}(z)\equiv \sin z$. 
For any complex number $z=s+it$, where $s$ and $t$ are real numbers and $i$ stands for the imaginary unit, let $\Re z:=s$ and $\Im z:=t$, the real and imaginary parts of $z$. 

In this paper, we shall present a number of identities involving certain integrals. 
For the sake of brevity, let us assume the following convention (unless specified otherwise):
when saying that such an identity takes place under certain conditions, we shall actually mean to say that under those conditions the corresponding integral exists in the Lebesgue sense (but may perhaps be infinite) \emph{and} the identity takes place. 

%For the sake of brevity, let us assume the following convention concerning the subsequent statements about identities involving integrals: when we say that such an identity takes place under certain conditions, it is meant, in particular, that under those conditions the corresponding integrals exist in the Lebesgue sense. 

\begin{theorem}\label{lem:ga}\ 
For any $s\in(0,s_2]$ and any $j=-1,0,\dots,\ell$ such that $\E|X|^{j_+}<\infty$, 
\begin{align}
 \E X_+^p &=
 \frac{\Ga(p+1)}{2\pi}\int_{-\infty}^\infty
\frac{\E e_j\big((s+it)X\big)}{(s+it)^{p+1}}\,\dd t \label{eq:}\\
&=
 \frac{\Ga(p+1)}{\pi}\int_0^\infty
\Re\frac{\E e_j\big((s+it)X\big)}{(s+it)^{p+1}}\,\dd t. \label{eq:Re}
\end{align} 
\end{theorem}

\begin{remark}\label{rem:details}
Of course, the statement of Theorem~\ref{lem:ga} is devoid of content in the case when $s_2=0$. As for the condition $\E|X|^{j_+}<\infty$, here we use the convention $0^0:=1$ (any other real number in place of $1$ would do here as well), to interpret $|X|^{j_+}$ when $X=0$ and 
%one obviously has $j_+=0$ for 
$j\in\{-1,0\}$. 
So, the condition $\E|X|^{j_+}<\infty$ will trivially hold if $j\in\{-1,0\}$. 
On the other hand, if $j\ge1$ (and hence $p>1$), then the expression $\E e_j\big((s+it)X\big)$ under the integrals on the right-hand sides of \eqref{eq:} and \eqref{eq:Re} would lose meaning without the condition $\E|X|^{j_+}<\infty$. 
\end{remark}

\begin{corollary}\label{cor:neg}\ 
If $p\in\N$ and $\E|X|^p<\infty$, then for any $s\in[s_1,0)$ and any $j=-1,0,\dots,\ell$% and $u\in\R$ one has
\begin{align}
 \E X_+^p &=\E X^p
+
\frac{p!}{2\pi}\,\int_{-\infty}^\infty
\frac{\E e_j\big((s+it)X\big)}{(s+it)^{p+1}}\,\dd t \label{eq:neg}\\
&=\E X^p
+
\frac{p!}{\pi}\,\int_0^\infty
\Re\frac{\E e_j\big((s+it)X\big)}{(s+it)^{p+1}}\,\dd t. \notag %\label{eq:neg Re}
\end{align}
\end{corollary}

\begin{corollary}\label{cor:0}\ 
If %$p\notin\N$ or $\E|X|^k<\infty$, 
$\E|X|^p<\infty$ then 
\begin{equation}\label{eq:0}
 \E X_+^p =
 \frac{\E X^k}2\ii{p\in\N}+
 \frac{\Ga(p+1)}\pi 
\,\int_0^\infty\Re
\frac{\E e_\ell(itX)}{(it)^{p+1}}\,\dd t.
\end{equation} 
In particular, one has the following: 
for any $m\in\N$ such that $\E X^{2m}<\infty$, 
\begin{equation}\label{eq:even}
 \E X_+^{2m} =
 \frac{\E X^{2m}}2+
 \frac{(2m)!}\pi 
\,\int_0^\infty
\frac{\E s_{2m-1}(tX)}{t^{2m+1}}\,\dd t;
\end{equation}
for any $m\in\N$ such that $\E|X|^{2m-1}<\infty$, 
\begin{equation}\label{eq:odd}
 \E X_+^{2m-1} =
 \frac{\E X^{2m-1}}2+
 \frac{(2m-1)!}\pi 
\,\int_0^\infty
\frac{\E c_{2m-2}(tX)}{t^{2m}}\,\dd t.
\end{equation}
\end{corollary}

\begin{corollary}\label{cor:XY}\ 
If $Y$ is another r.v. then 
\begin{equation}\label{eq:XY}
 \E X_+^p =\E Y_+^p +
 \frac{\Ga(p+1)}{\pi}\int_0^\infty\Re
\frac{\E e^{itX}-\E e^{itY}}{(it)^{p+1}}\,\dd t 
\end{equation}
provided that $\E|X|^p+\E|Y|^p<\infty$ and 
$\E X^j=\E Y^j$ for all $j=1,\dots,k$.
 
Moreover, one has the following: 
for any $m\in\N$
\begin{equation}\label{eq:even_XY}
 \E X_+^{2m} =\E Y_+^{2m} +(-1)^m\,
 \frac{(2m)!}{\pi}\int_0^\infty
\frac{\E\sin tX-\E\sin tY}{t^{2m+1}}\,\dd t 
\end{equation}
if $\E X^{2m}+\E Y^{2m}<\infty$ and 
$\E X^{2j+1}=\E Y^{2j+1}$ for all $j=0,\dots,m-1,m-\frac12$; 
for any $m\in\N$
\begin{equation}\label{eq:odd_XY}
 \E X_+^{2m-1} =\E Y_+^{2m-1} +(-1)^m\,
 \frac{(2m-1)!}{\pi}\int_0^\infty
\frac{\E\cos tX-\E\cos tY}{t^{2m}}\,\dd t 
\end{equation}
if $\E|X|^{2m-1}+\E|Y|^{2m-1}<\infty$ and 
$\E X^{2j}=\E Y^{2j}$ for all $j=0,\dots,m-1,m-\frac12$. 
\end{corollary}

For any r.v.\ $X$ as in Corollary~\ref{cor:0}, it is always possible (and easy) to construct a r.v.\ $Y$ as in Corollary~\ref{cor:XY} such that the characteristic function $\E e^{itY}$, and hence its real and imaginary parts $\E\cos tY$ and $\E\sin tY$, are easily computable. In particular, one can always take $Y$ to be a r.v.\ with finitely many values; more specifically, $k+1$ values would suffice for \eqref{eq:XY} and $m+2$ values for \eqref{eq:even_XY} or \eqref{eq:odd_XY}. 
An obvious advantage of the formulas given in Corollary~\ref{cor:XY} (over those in Corollary~\ref{cor:0}) is that the corresponding integrals will converge (for $p>1$) faster near $\infty$, and just as fast near $0$. 
Also, Corollary~\ref{cor:XY} will be just one step closer than Corollary~\ref{cor:0} to possible applications to the convergence of the positive-part moments in the central limit theorem; see Example~\ref{ex:CLT} in Section~\ref{appls} for further details.

%\begin{ex}\label{ex:CLT} 
%Just as it was done in \cite{brown70} for absolute moments, the identities given in Corollary~\ref{cor:XY} could be used to prove the convergence of the positive-part moments in the central limit theorem, or even to obtain bounds on the rate of such convergence. 
%However, we shall not be pursuing this matter here, since the main motivation for the present paper was the need for an effective computation of the bound $\Pin(x)$, as described in Example~\ref{ex:hoeff}.   
%\end{ex}

\begin{remark}\label{rem:left}
If $\E X_+^p<\infty$ for an even natural $p=2m$, then it is clear that identity \eqref{eq:0} (or, equivalently, \eqref{eq:even}) cannot hold without the condition $\E|X|^p<\infty$, since one needs the moment $\E X^k=\E X^p$ on the right-hand side of \eqref{eq:0} to exist.  
Similarly, if $\E X_+^p<\infty$ for an odd natural $p=2m-1$, then identity \eqref{eq:0} (or, equivalently, \eqref{eq:odd}) cannot hold without the condition $\E|X|^p<\infty$. 

However, if $p>0$ is not an integer, the term $\frac{\E X^k}2\ii{p\in\N}$ on the right-hand side of \eqref{eq:0} disappears. One may then wonder as to what, if any, moment condition on the \emph{left} tail of the distribution of $X$ is needed in order for identity \eqref{eq:0} to hold -- having in mind that the finiteness of the positive part moment $\E X_+^p$ depends only on the \emph{right} tail of the distribution of $X$. 
Perhaps surprisingly, it turns out that a left-tail condition %close to the condition $\E X_-^p<\infty$ 
is necessary for \eqref{eq:0}, even if the integral in \eqref{eq:0} is allowed to be understood as an improper one (at $0$); moreover, 
the moment condition $\E|X|^p<\infty$ in Corollary~\ref{cor:0} can be only slightly relaxed, even when $p$ is not an integer.  
Actually, one can obtain a necessary and sufficient condition for such a representation of the positive-part moment in terms of the characteristic function, as described in the following theorem.  
\end{remark}

\begin{theorem}\label{cor:improp}\ 
Take any $p\in(0,\infty)\setminus\N$ and any r.v.\ $X$ with $\E X_+^p<\infty$. 
Then  
the following two conditions are equivalent to each other:
\begin{enumerate}%[(I)] 
\item[\emph{(I)}]\label{improp} $\E|X|^\ell<\infty$\quad and\quad  
$%\begin{equation*}%\label{eq:improp}
 {\displaystyle \E X_+^p =
 \dfrac{\Ga(p+1)}\pi 
\,\int_{0+}^\infty\Re
\dfrac{\E e_\ell(itX)}{(it)^{p+1}}\,\dd t;}
$%\end{equation*} 
\item[\emph{(II)}]\label{improp equiv}
$%\begin{equation*}%\label{eq:improp equiv}
 {\displaystyle \PP(X_->x)=o(1/x^p)\ \text{as}\ x\to\infty.}
$%\end{equation*} 
\end{enumerate}
\end{theorem}

Here and in what follows, $x_-:=(-x)_+$ for all $x\in\R$; also let $x_-^p:=(x_-)^p$. 

%\begin{remark}\label{rem:details}
%Of course, the statement of Theorem~\ref{lem:ga} is devoid of content in the case when $s_2=0$. As for the condition $\E|X|^{j_+}<\infty$, here we use the convention $0^0:=0$, to interpret $|X|^{j_+}$ when $X=0$ and 
%%one obviously has $j_+=0$ for 
%$j\in\{-1,0\}$. 
%So, the condition $\E|X|^\ell<\infty$ will trivially hold if $j\in\{-1,0\}$. 
%On the other hand, if $j\ge1$ (and hence $p>1$), then the expression $\E e_j\big((s+it)X\big)$ under the integrals on the right-hand sides of \eqref{eq:} and \eqref{eq:Re} would lose meaning without the condition $\E|X|^{j_+}<\infty$. 
%\end{remark}

Note that the part $\E|X|^\ell<\infty$ of condition (I) of Theorem~\ref{cor:improp} will trivially hold if $0<p<1$ (and hence $\ell=0$). 
On the other hand, if $p>1$ (and hence $\ell\ge1$), then the expression $\E e_\ell(itX)$ in condition (I) of Theorem~\ref{cor:improp} would lose meaning if $\E|X|^\ell=\infty$; cf.\ Remark~\ref{rem:details}. 

Note also that,  
%\begin{proposition}\label{prop:lower mom}\ 
if condition~(II) or, equivalently, (I) of Theorem~\ref{cor:improp} holds, then $\E X_-^r<\infty$ for all $r\in(0,p)$. 
Indeed, 
%\end{proposition}
condition~(II) of Theorem~\ref{cor:improp} implies 
\begin{equation}\label{eq:rth moments}
\E X_-^r={\scriptstyle\int_0^\infty} rx^{r-1}\,\PP(X_->x)\,\dd x
=O\big({\scriptstyle\int_0^\infty} rx^{r-1}\,(1\wedge x^{-p})\,\dd x\big)<\infty.	
\end{equation}

On the other hand, it is easy to give examples where condition~(II) of Theorem~\ref{cor:improp} holds, while $\E X_-^p=\infty$, and so,  Corollary~\ref{cor:0} is not applicable; for instance, take any r.v.\ $X$ such that $\PP(X_->x)=1/(x^p\ln x)$ for all large enough $x>0$. 

In view of Corollary~\ref{cor:0} and Theorem~\ref{cor:improp}, 
for the identity \eqref{eq:0} to hold for a given $p\in(0,\infty)\setminus\N$ (with the integral understood in the Lebesgue sense) and a r.v.\ $X$ with $\E X_+^p<\infty$, it is sufficient that 
$\E X_-^p<\infty$ and it is necessary that $\E X_-^r<\infty$ for all $r\in(0,p)$. 

Yet, no ``simple'' necessary \emph{and} sufficient condition for \eqref{eq:0} to hold -- in the Lebesgue sense -- for a given $p\in(0,\infty)\setminus\N$ appears to be known; it is apparently an open question whether such a condition exists at all. 
However, one can see that condition (II) of Theorem~\ref{cor:improp} is not sufficient (although, by Theorem~\ref{cor:improp}, it is necessary) for \eqref{eq:0}. 
Indeed, for   
%
%\begin{proposition}\label{prop:improp}\ 
%For 
any $p\in(0,\infty)\setminus\N$, there exists a r.v.\ $X$ with $\E X_+^p<\infty$ such that the two equivalent conditions -- (I) and (II) -- of Theorem~\ref{cor:improp} hold, while the integral 
$\int_{0}^\infty\Re\frac{\E e_\ell(itX)}{(it)^{p+1}}\,\dd t$ does not exist in the Lebesgue sense. 
%\end{proposition}
For instance, one can consider 
\begin{ex}\label{ex:improp}
The idea is to take a r.v.\ $X$ with a lacunary distribution and then construct a matching lacunary set $A\subset(0,\infty)$ such that the integrand 
$\Re\frac{\E e_\ell(itX)}{(it)^{p+1}}$ is of constant sign and large enough in absolute value for $t\in A$. 
Take indeed any $p\in(0,\infty)\setminus\N$ and let $X$ be a discrete r.v.\ such that, for some $d\in(1,\infty)$, one has 
$%\begin{equation}
	\PP(X=-d^n)=\frac c{nd^{np}}\quad\text{for all } n\in\N,
$ %\end{equation}
where $c$ is the constant chosen so that $\sum_{n\in\N}\PP(X=-d^n)=1$. Then 
it is easy to check that condition (II) (and hence condition (I)) of Theorem~\ref{cor:improp} holds. 
Moreover, $X_+=0$ a.s., and so, $\E X_+^p<\infty$. 
Further, one can show that there exist some $d=d(p)>1$, $\vp=\vp(p)\in(0,1-1/d)$, and $m=m(p)\in\N$ such that 
$%\begin{equation*}%\label{eq:asymp final}
	(-1)^m\Re\frac{\E e_\ell(itX)}{(it)^{p+1}}\ge
	\frac1{t\ln\frac\vp t}
$ %\end{equation*}
for all $t\in A:=\bigcup\nolimits_{n=1}^\infty[\frac{\vp(1-\vp)}{d^n},\frac\vp{d^n}]$, whence 
$(-1)^m\int_A\Re\frac{\E e_\ell(itX)}{(it)^{p+1}}\,\dd t=\infty$. 
For details, see \cite[Proposition~2.6 and its proof]{pos-arxiv}. \qed
\end{ex}

Theorem~\ref{cor:improp} 
and Example~\ref{ex:improp} 
%Proposition~\ref{prop:improp} 
are similar in spirit to 
a number of known if-and-only-if conditions and counterexamples, respectively; those results in the literature mainly concern such simpler relations as the ones of the tails/moments of the distribution with the derivatives of the characteristic function at $0$ (but also with other linear functionals of the c.f.) 
-- see e.g.\ \cite{fortet,zygmund47,pitman56,boas67,boas68}. 
%
%There are a number of known if-and-only-if conditions and counterexamples similar in spirit, respectively, to Theorem~\ref{cor:improp} and Proposition~\ref{prop:improp}; those results in the literature mainly concern such simpler relations as the ones of the tails/moments of the distribution with the derivatives of the characteristic function at $0$ (but also with other linear functionals of the c.f.) 
%-- see e.g.\ \cite{fortet,zygmund47,pitman56,boas67,boas68}. 
%
Ibragimov \cite{ibrag66} provides necessary and sufficient conditions -- both in terms of (truncated) moments/tails and the c.f.\ -- for rates $O(n^{-p})$ in the central limit theorem to hold; cf.\ Example~\ref{ex:CLT} in the next section. 

%!!!!!!!!!!!! similar results in the literature !!!!!!!!!!!!!

\section{Applications}\label{appls}
Here let us give examples of known or potential applications of identities presented in Section~\ref{results}. 
\begin{ex}\label{ex:hoeff} 
Let $X_1,\dots,X_n$ be independent r.v.'s such that $X_i\le y$ for some real $y>0$ and $\E X_i\le0$, for all $i$. Let 
$S:=X_1+\dots+X_n$ and assume that $\sum_i\E X_i^2\le\si^2$ for some real $\si>0$. Also, let 
$\vp:=\frac1{\si^2y}\,\sum_i\E(X_i)_+^3$, so that $0<\vp<1$. 
In the paper \cite{pin08}, mentioned in the Introduction, an optimal in a certain sense upper bound on the tail $\PP(S\ge x)$ was obtained, of the form 
\begin{equation*}%\label{eq:P(x)}
\Pin(x):=\frac{\E^3(\eta-t_x)_+^2}{\E^2(\eta-t_x)_+^3},  	
\end{equation*}
where $x\in\R$; $\eta:=\Ga_{(1-\vp)\si^2}+y\tPi_{\vp\si^2/y^2}$; 
$\Ga_{(1-\vp)\si^2}$ and $\tPi_{\vp\si^2/y^2}$ are any independent r.v.'s whose distributions are, respectively, the centered Gaussian one with variance $(1-\vp)\si^2$ and the centered Poisson one with variance $\vp\si^2/y^2$; $t_x\in\R$ is the only real root of the equation $m(t_x)=x$; and $m(t):=t+\E(\eta-t)_+^3/\E(\eta-t)_+^2$. 
Thus, to compute the bound $\Pin(x)$, one needs a fast calculation of the positive-part moments of the r.v.\ $X:=\eta-t$. While the direct approaches mentioned in the Introduction do not work here, the formulas of Theorem~\ref{lem:ga} and Corollary~\ref{cor:0} prove very effective, since the Fourier-Laplace transform of the r.v.\ $\eta-t$ is given by a simple expression: 
\begin{equation}\label{eq:PUexp}
	\E e^{z(\eta-t)}=\exp\big\{-zt+\tfrac{\la^2}2\,(1-\vp)\si^2+\tfrac{e^{zy}-1-zy}{y^2}\,\vp\si^2\big\}\quad
	\forall z\in\C. 
\end{equation}
It takes under $0.5$ sec in Mathematica on a standard Core 2 Duo laptop to produce an entire graph of the bound $\Pin(x)$ for a range of values of $x$ using either \eqref{eq:Re} or \eqref{eq:0}. 
%As mentioned in the Introduction, the need for such an effective computation of the bound $\Pin(x)$ was the main motivation for the present paper. 
\end{ex}

\begin{ex}\label{ex:finance} 
As pointed out by a referee, positive-part moments naturally arise in mathematical finance. For example, $(S-K)_+$ is the value of a call option with the strike price $K$ when the stock price is $S$. 
Exact bounds on $\E(S-K)_+$ under various conditions were recently given in \cite{de la pena et al}. 
\end{ex}

\begin{ex}\label{ex:CLT} 
Just as it was done in \cite{brown70} for absolute moments, the identities given in Corollary~\ref{cor:XY} could be used to prove the convergence of the positive-part moments in the central limit theorem, or even to obtain bounds on the rate of such convergence. 
However, we shall not be pursuing this matter here, since the main motivation for the present paper was the need for an effective computation of the bound $\Pin(x)$, as described in Example~\ref{ex:hoeff}.   
\end{ex}

\section{Proofs}\label{proofs}
For any two expressions $a$ and $b$, let us write: $a\OO b$ or, equivalently, $b\OOG a$ if $|a|\le Cb$ for some positive constant $C$ depending only on $p$; $a\asymp b$ if $ab>0$, $a\OO b$, and $b\OO a$.

Let us also write  
$a<<b$ (to be read ``$a$ is much less than $b$'') or, equivalently, $b>>a$ (``$b$ is much greater than $a$'') if $b>0$ and $|a|=o(b)$. Accordingly, if $\mathcal A$ is an assertion and $a>0$, then ``$\mathcal A$ holds if $a<<b$'' will mean ``for all $p\in(0,\infty)\setminus\N$ there exists some constant $K=K(p)>0$ such that $\mathcal A$ holds whenever $|a|<b/K$'', that is, $\mathcal A$ will hold ``eventually''.
Similarly, ``if $a<<b$ and $\mathcal A$ holds, then $a_1<<b_1$'' will mean ``for all $p\in(0,\infty)\setminus\N$ and all $K_1>0$ there exists $K=K(p,K_1)>0$ such that one has the implication `if $|a|<b/K$ and $\mathcal A$ holds, then $|a_1|<b_1/K_1$'\;''.

Also, let us write $a\sim b$ if $a/b\to1$. 

\begin{proof}[Proof of Theorem~\ref{lem:ga}]\ 
Note that \eqref{eq:Re} immediately follows from \eqref{eq:}, since the integrand in \eqref{eq:Re} is even in $t\in\R$. 
Note also that 
$\int_{C_s}\frac{z^j}{z^{p+1}}\,\dd z=0$ for all $j=0,\dots,\ell$ and $s\ne0$, where $C_s:=\{z\in\C\colon\Re z=s\}$. 
So, it is enough to prove \eqref{eq:} %and \eqref{eq:Re} 
with $j=-1$. 
The key observation here is the following: for any $x\in\R$ and any complex number $z=s+it$ with $s:=\Re z>0$, 
\begin{equation*}%\label{eq:lap}
\int_{-\infty}^\infty e^{zu}\,(x-u)_+^p\,\dd u
%=\int_{-\infty}^\infty e^{zu}\,(x-u)_+^p\,\dd u
=\int_{-\infty}^x e^{zu}\,(x-u)^p\,\dd u %\\
=e^{zx}\int_0^\infty v^p e^{-zv}\,\dd v
=\frac{\Ga(p+1)}{z^{p+1}}\,e^{zx};
\end{equation*}
this is obvious for real $z>0$, and then one can use %it can be easily extended to all complex $z$ with $\Re z>0$ by 
analytic continuation. 
%
%the last equality here follows by the Cauchy integral theorem, applied to the function $w\mapsto w^p e^{-w}$ and the (appropriately parameterized) closed contour 
%$\{\vp e^{i\th}\colon0\le\th\le\arg z\}
%\cup\{u\in\R\colon\vp\le u\le R\}
%\cup\{R e^{i\th}\colon0\le\th\le\arg z\}
%\cup\{zv\colon\vp\le v\le R\}$, 
%with $\vp\downarrow0$ and $R\uparrow\infty$; by complex conjugation, without loss of generality (w.l.o.g.) one may assume in \eqref{eq:lap} that $\Im z>0$.
Thus, for each $s\in(0,\infty)$ the Fourier transform $\R\ni t\longmapsto\int_{-\infty}^\infty e^{itu}\,e^{su}\,(x-u)_+^p\,\dd u$ %\break 
of the integrable function $\R\ni u\longmapsto e^{su}\,(x-u)_+^p$ is the function $\R\ni t\longmapsto\break
\Ga(p+1)\,\frac{ e^{(s+it)x}}{(s+it)^{p+1}}$, which is integrable as well. So, by the inverse Fourier transform, one obtains \eqref{eq:} (with $j=-1$) with the real constant $x$ in place of the r.v.\ $X$. (Equivalently, one can use here the Laplace inversion.) 
Now \eqref{eq:} follows by the Fubini theorem, since 
$%\begin{equation}\label{eq:O}
\E\big|\frac{e^{(s+it)X}}{(s+it)^{p+1}}\big|=\frac{\E e^{sX}}{(s^2+t^2)^{(p+1)/2}}
$ %\end{equation}
for all $s>0$ and $t\in\R$. %, and $j=-1,0,\dots,\ell$. 
%In turn, \eqref{eq:Re} immediately follows from \eqref{eq:}, since the integrand in \eqref{eq:Re} is even in $t\in\R$. 
\end{proof}

\begin{proof}[Proof of Corollary~\ref{cor:neg}]\  
Here, just as in the proof of Theorem~\ref{lem:ga}, without loss of generality (w.l.o.g.)\ $j=-1$. 
Consider first the case when $s_2>0$. W.l.o.g., $s_1<0$. 
By the convexity of $\E e^{sX}$ in $s$, one has  
$\big|\frac{\E e^{(s+it)X}}{(s+it)^{p+1}}\big|\le
\frac{\E e^{s_1X}\vee\E e^{s_2X}}{|t|^{p+1}}$ for all $s\in[s_1,s_2]$ and all $t\ne0$, so that 
$\int_{s_1}^{s_2}\big|\frac{\E e_j\big((s+it)X\big)}{(s+it)^{p+1}}\big|\,\dd s\to0$ as $|t|\to\infty$. % \big(cf.\ \eqref{eq:O}\big). 
Hence, by the Cauchy integral theorem and the continuity of the integral $I(s):=\int_{-\infty}^\infty\frac{\E e^{(s+it)X}}{(s+it)^{p+1}}\,\dd t$ in $s\in[s_1,s_2]$, the difference between the value of $I(s)$ for $s\in(0,s_2]$ and that of $I(s)$ for $s\in[s_1,0)$ is 
$%\begin{equation}
	\frac1i\int_{C_\vp}\frac{\E e^{zX}}{z^{p+1}}\,\dd z
$, %\end{equation}
where $C_\vp$ is the circle in $\C$ of radius $\vp$ centered at $0$ and traced out counterclockwise, provided that $\vp\in(0,s_1\wedge s_2)$. 
So, \eqref{eq:neg} (in the case when $s_2>0$ and with $j=-1$) follows from \eqref{eq:} by taking $\vp\downarrow0$, since $e^{zX}=e_p^{(zX)}+O(|zX|^{p+1}e^{\vp|X|})$ over $z\in C^\vp$; recall that here it is assumed that $p\in\N$. 

Assume now that $s_2=0$ (and, again, $j=-1$). Then, by what has just been proved, one has \eqref{eq:neg} for any $s\in[s_1,0)$ with $X\wedge N$ in place of $X$, where $N$ is any real number. 
By the dominated convergence theorem and in view of the condition $\E|X|^p<\infty$, one has 
$\E(X\wedge N)_+^p\to\E X_+^p$ and $\E(X\wedge N)^p\to\E X^p$ as $N\to\infty$. 
To complete the proof of Corollary~\ref{cor:neg}, it remains to use again the dominated convergence theorem, in view of the inequalities $\big|\frac{\E e^{(s+it)(X\wedge N)}}{(s+it)^{p+1}}\big|\le
\frac{\E e^{s(X\wedge N)}}{(s^2+t^2)^{(p+1)/2}}$ and  $e^{s(X\wedge N)}%\le e^{s(X\wedge0)}=1\vee e^{sX}\le1\vee(1\vee e^{s_1X})%\vee e^{s_1(X\wedge N)}
\le1+e^{s_1X}$ for all $N\ge0$, $s\in[s_1,0)$, and $t\in\R$.
\end{proof}

\begin{proof}[Proof of Corollary~\ref{cor:0}]\  
Let us first prove identity \eqref{eq:0} in the case when the r.v.\ $X$ takes on only one value, say $x\in\R$.
For any $\vp>0$, consider the integration contour $C:=\{it\colon t\le-\vp\}\cup C_\vp^+\cup\{it\colon t\ge\vp\}$, where  $C_\vp^+:=\{z\in\C\colon|z|=\vp,\ \Re z>0\}$.
Then, by \eqref{eq:} (with $j=-1$) and the Cauchy integral theorem, %for all $\vp>0$
\begin{equation}\label{eq:pre 0,x}
 \frac{2\pi i}{\Ga(p+1)}\,x_+^p 
 =\int_C\frac{e^{zx}}{z^{p+1}}\,\dd z 
 =\int_C\frac{e_\ell(zx)}{z^{p+1}}\,\dd z,
\end{equation}
because $\int_C\frac{z^j}{z^{p+1}}\,\dd z=0$ for all $j=0,\dots,\ell$.
Next, $e_\ell(zx)=\frac{z^{\ell+1}}{(\ell+1)!}x^{\ell+1}+O(|z|^{\ell+2})$ as $z\to0$, and so, 
\begin{align*}
	\int_{C_\vp^+}\frac{e_\ell(zx)}{z^{p+1}}\,\dd z
	&=\frac{x^{\ell+1}}{(\ell+1)!}\int_{C_\vp^+}\frac{\dd z}{z^{p-\ell}}+O(\vp^{\ell-p+2}) \\
	&=i\pi\frac{x^{\ell+1}}{(\ell+1)!}\ii{p\in\N}
	+O\big(\vp^{\ell-p+1}(\ii{p\notin\N}+\vp)\big) \\
	&=i\pi\frac{x^{\ell+1}}{(\ell+1)!}\ii{p\in\N}
	+o(1)
	=i\pi\frac{x^k}{\Ga(p+1)}\ii{p\in\N}
	+o(1) 
\end{align*}
as $\vp\downarrow0$. 
This and \eqref{eq:pre 0,x} imply
\begin{equation}\label{eq:0,x}
 x_+^p =
 \frac{x^k}2\ii{p\in\N}+
 \frac{\Ga(p+1)}\pi 
\,\int_{0+}^\infty\Re
\frac{e_\ell(itx)}{(it)^{p+1}}\,\dd t, 
\end{equation}
since the latter integrand is even in $t\in\R$. 

Next, observe that 
\begin{equation}\label{eq:e estimate}
|e_\ell(iu)|\OO|u|^\ell\wedge|u|^{\ell+1}	
\end{equation}
over all $u\in\R$, and so, for all $x\in\R$
\begin{equation}\label{eq:abs}
	\int_0^\infty\Big|
\frac{e_\ell(itx)}{(it)^{p+1}}\Big|\dd t=c_\ell|x|^p,
\end{equation}
where $c_\ell:=\int_0^\infty|\frac{e_\ell(iu)}{(iu)^{p+1}}|\dd u<\infty$
for $p\in(0,\infty)\setminus\N$. 
Therefore, \eqref{eq:0} follows from \eqref{eq:0,x} by the Fubini theorem -- provided that $p\in(0,\infty)\setminus\N$. 

The case $p\in\N$ is treated quite similarly. Namely, if $p=2m$ for some $m\in\N$, then 
$\Re\frac{e_\ell(itx)}{(it)^{p+1}}=\frac{s_{2m-1}(tx)}{t^{2m+1}}$ for all $x\in\R$ and $t>0$, and so, 
\eqref{eq:even} follows by the Fubini theorem from \eqref{eq:0,x}, the estimate $|s_{2m-1}(u)|\OO|u|^{2m-1}\wedge|u|^{2m+1}$ for all $u\in\R$, and the condition $\E X^{2m}<\infty$. If $p=2m-1$ for some $m\in\N$, then \eqref{eq:odd} similarly follows using the estimate $|c_{2m-2}(u)|\OO|u|^{2m-2}\wedge|u|^{2m}$ for all $u\in\R$ and the condition $\E|X|^{2m-1}<\infty$. 
\end{proof}

\begin{proof}[Proof of Corollary~\ref{cor:XY}]\ 
This immediately follows from Corollary~\ref{cor:0}, applied the r.v.\ $Y$, as well as to $X$. (Note that $k=p$ and $\ell=p-1$ if $p\in\N$, and $k=\ell$ if $p\notin\N$.)
\end{proof}

\begin{proof}[Proof of Theorem~\ref{cor:improp}]\ 
\emph{Step 0.}\quad 
Take indeed any $p\in(0,\infty)\setminus\N$ and any r.v.\ $X$ with $\E X_+^p<\infty$. 
Recalling \eqref{eq:rth moments}, note that either one of the conditions (I) and (II) implies  
\begin{equation}\label{eq:EX ell}
	 \E|X|^\ell<\infty.
\end{equation}

\emph{Step 1.}\quad 
For $v\ge0$, let
\begin{equation}\label{eq:I}
	I(v):=I_p(v)
	:=\int_v^\infty\Re\frac{ e_\ell(-iu)}{(iu)^{p+1}}\,\dd u. 
\end{equation}
This definition is correct, since the integral exists in the Lebesgue sense, in view of estimate \eqref{eq:e estimate}. 
Note that
\begin{equation}\label{eq:I(0)}
	I_p(0)=0;
\end{equation}
this is a special case of \eqref{eq:0}, with $X=-1$ almost surely (a.s.).  
In turn, \eqref{eq:I(0)} implies
\begin{equation}\label{eq:-I}
	I_p(v)=-\int_0^v\Re\frac{ e_\ell(-iu)}{(iu)^{p+1}}\,\dd u. 
\end{equation}

\emph{Step 2.}\quad At this step, let us prove that 
if $p\in(0,1)$ and $v\in(0,\infty)$, then 
\begin{equation}\label{eq:I>0}
	I_p(v)>0.
\end{equation}
Take indeed any $p\in(0,1)$. Then $\ell=0$, and 
\eqref{eq:I} and \eqref{eq:-I} can be rewritten as 
\begin{align}
	I(v)&=\int_v^\infty
	\frac{\sin\frac{\pi p}2-\sin(\frac{\pi p}2+u)}{u^{p+1}}\,\dd u \label{eq:simplif}\\
	&=\int_0^v
	\frac{\sin(\frac{\pi p}2+u)-\sin\frac{\pi p}2}{u^{p+1}}\,\dd u \label{eq:simplif-}
\end{align}
for all $v>0$. 
Observe next that $v^{p+1}I'(v)=\sin(\frac{\pi p}2+v)-\sin\frac{\pi p}2$ for all $v>0$, whence the only local minima of $I$ in $(0,\infty)$ are at the points $2\pi,4\pi,\dots$. 
Therefore, and because of \eqref{eq:I(0)} and the obvious equality $I(\infty-)=0$, to prove \eqref{eq:I>0} it suffices to show that $I(2j\pi)>0$ for all $j\in\N$. 

Using \eqref{eq:simplif} and (twice) applying the Fubini theorem, for all $j\in\N$ and $q\in\N$ such that $j<q$ one has %the following (for some positive constants $c_1,c_2,\dots$): 
\begin{align*}
I(2j\pi)-I(2q\pi)&=\int_{2j\pi}^{2q\pi}
	\frac{\sin(\frac{\pi p}2+2j\pi)-\sin(\frac{\pi p}2+u)}
	{u^{p+1}}\,\dd u \\
&=-\int_{2j\pi}^{2q\pi} \frac{\dd u}{u^{p+1}}\,
\int_{2j\pi}^u\dd t\,\cos\Big(\frac{\pi p}2+t\Big) \\
&=-\int_{2j\pi}^{2q\pi}\dd t\,\cos\Big(\frac{\pi p}2+t\Big)
\int_t^{2q\pi} \frac{\dd u}{u^{p+1}}\,\\
&=-\frac1p\,\int_{2j\pi}^{2q\pi}\dd t\,\cos\Big(\frac{\pi p}2+t\Big)\,
\Big(\frac1{t^p}-\frac1{(2q\pi)^p}\Big)\,\\
&=-\frac1p\,\int_{2j\pi}^{2q\pi}\dd t\,\cos\Big(\frac{\pi p}2+t\Big)\,
\frac1{t^p}\,\\
&=-\frac1{p\Ga(p)}\,\int_{2j\pi}^{2q\pi}\dd t\,\cos\Big(\frac{\pi p}2+t\Big)\,
\int_0^\infty u^{p-1}e^{-tu}\,\dd u\\
&=-\frac1{\Ga(p+1)}\,\int_0^\infty u^{p-1}\,\dd u
\int_{2j\pi}^{2q\pi}\dd t\,\cos\Big(\frac{\pi p}2+t\Big)\,e^{-tu}\\
&=\frac1{\Ga(p+1)}\,\int_0^\infty u^{p-1}\,
\frac{\sin\frac{\pi p}2-u\cos\frac{\pi p}2}{1+u^2}\,
(e^{-2j\pi u}-e^{-2q\pi u})\,\dd u;
\end{align*}
letting now $q\to\infty$ and using again the equality  $I(\infty-)=0$, by the dominated convergence and a comparison of the integrands one obtains 
\begin{align}
I(2j\pi)&=\frac1{\Ga(p+1)}\,\int_0^\infty u^{p-1}\,
\frac{\sin\frac{\pi p}2-u\cos\frac{\pi p}2}{1+u^2}\,
e^{-2j\pi u}\,\dd u \notag\\
&\ge\frac1{\Ga(p+1)}\,\int_0^\infty u^{p-1}\,
\frac{\sin\frac{\pi p}2-u\cos\frac{\pi p}2}{1+\tan^2\frac{\pi p}2}\,
e^{-2j\pi u}\,\dd u \notag\\
&=\frac{(2j\pi\tan\frac{\pi p}{2}-p)\,\cos^3\frac{\pi p}{2}}
{(2j\pi)^{p+1}p}
\ge\frac{(j\pi^2-1)\,\cos^3\frac{\pi p}{2}}
{(2j\pi)^{p+1}}>0,\label{eq:>0}
\end{align}
which completes the proof of \eqref{eq:I>0}.  

\emph{Step 3.}\quad 
Still assuming that $p\in(0,1)$, it follows from \eqref{eq:>0} that $I(2j\pi)\ge\frac{(j\pi^2-1)\,\cos^3\frac{\pi p}{2}}
{(2j\pi)^{p+1}}\OOG j^{-p}$ over all $j\in\N$. 
Next, for every $j\in\N$, the only local minima of $I$ on the interval  $[2j\pi,(2j+2)\pi]$ are the endpoints; so, $I(v)\ge I(2j\pi)\wedge I\big((2j+2)\pi\big)\OOG j^{-p}\OOG v^{-p}$ over all  $v\in[2j\pi,(2j+2)\pi]$ and all $j\in\N$. That is, $I(v)\OOG v^{-p}$ over all  $v\in[2\pi,\infty)$. 
On the other hand, by \eqref{eq:simplif}, $|I(v)|\le\int_v^\infty
\frac2{u^{p+1}}\,\dd u\OO v^{-p}$ over all $v\in(0,\infty)$. 
Thus, $I(v)\asymp v^{-p}$ over all $v\in[2\pi,\infty)$. 

Yet on the other hand, by \eqref{eq:simplif-}, $I(v)\asymp\int_0^v
\frac u{u^{p+1}}\,\dd u\asymp v^{1-p}$ over all $v$ in a right neighborhood of $0$. 

Also, $I_p$ is obviously continuous on $(0,\infty)$. So, in view of \eqref{eq:I>0}, for each $p\in(0,1)$, 
\begin{equation}\label{eq:I, p<1}
	I_p(v)\asymp v^{-p}(v\wedge1)\quad\text{over all $v>0$.}
\end{equation}

\emph{Step 4.}\quad Here, let us show, by induction on $\ell$,  that for each $p\in(0,\infty)\setminus\N$
\begin{equation}\label{eq:J}
	(-1)^\ell J_p(x,v)\asymp v^{-p}\big((vx)^{\ell+1}\wedge(vx)^\ell\big)
\end{equation}
over all $v>0$ and $x>0$, where
\begin{equation}\label{eq:J def}
	J_p(x,v):=\int_v^\infty\Re\frac{ e_\ell(-itx)}{(it)^{p+1}}\,\dd t=x^p\,I_p(vx)  
\end{equation}
for all $v>0$ and $x\ge0$. 
For $\ell=0$ (that is, for $p\in(0,1)$), \eqref{eq:J} is equivalent to \eqref{eq:I, p<1}. 
Assume now that $\ell\ge1$ and \eqref{eq:J} holds for $\ell-1$ and $p-1$ instead of $\ell$ and $p$, respectively. Then, by \eqref{eq:abs} and Fubini's theorem, for all $v>0$ and $x>0$
\begin{equation}\label{eq:J induction}
\begin{aligned}
	(-1)^\ell J_p(x,v)
	&=(-1)^{\ell-1}\int_0^x J_{p-1}(u,v)\,\dd u \\
	&\asymp\int_0^x v^{-(p-1)}\big((vu)^\ell\wedge(vu)^{\ell-1}\big)\,\dd u
	=v^{-p}\tilde J_\ell(vx),
\end{aligned}	
\end{equation}
where $\tilde J_\ell(t):=\int_0^t (z^\ell\wedge z^{\ell-1})\,\dd z$ for $t>0$, so that $\tilde J_\ell(t)\asymp t^{\ell+1}$ over all $t$ in a right neighborhood of $0$,  $\tilde J_\ell(t)\asymp t^\ell$ over all $t$ in a left neighborhood of $\infty$, and  $\tilde J_\ell>0$ on $(0,\infty)$. Therefore, $\tilde J_\ell(t)\asymp t^{\ell+1}\wedge t^\ell$ over all $t>0$. This and \eqref{eq:J induction} yield \eqref{eq:J}. 

\emph{Step 5.}\quad In view of \eqref{eq:J def} and
since $e_\ell(itx)=e_\ell(itx_+)+e_\ell(-itx_-)$ for all real $t$ and $x$, one has 
\begin{equation}\label{eq:step5}
	\int_v^\infty\Re\frac{\E e_\ell(itX)}{(it)^{p+1}}\,\dd t
	=\int_v^\infty\Re\frac{\E e_\ell(itX_+)}{(it)^{p+1}}\,\dd t
	+\E J_p(X_-,v) 
\end{equation}
for all $v>0$ by Fubini's theorem, which is applicable because of \eqref{eq:e estimate} and \eqref{eq:EX ell}. 
By \eqref{eq:0} with $X_+$ in place of $X$, the integral on the right-hand side of \eqref{eq:step5} converges to $\frac\pi{\Ga(p+1)}\,\E X_+^p$ as $v\downarrow0$. 

So, condition (I) of Theorem~\ref{cor:improp} is equivalent to $\E J_p(X_-,v)\to0$ as $v\downarrow0$, which in turn is equivalent, by \eqref{eq:J}, to $v^{-p}\E\big((vX_-)^{\ell+1}\wedge(vX_-)^\ell\big)\to0$ as $v\downarrow0$ and then to 
%\begin{equation}
%x^p\E\big((\frac{X_-}x)^{\ell+1}\wedge(\frac{X_-}x)^\ell\big)
%	\longrightarrow0\quad\text{as}\quad x\to\infty.
%\end{equation}
\begin{equation}\label{eq:ii}
x^{p-\ell-1}\E X_-^{\ell+1}\ii{X_-\le x}+x^{p-\ell}\E X_-^\ell\ii{X_->x}
	\underset{x\to\infty}\longrightarrow0.
\end{equation}

\emph{Step 6.}\quad Here we complete the proof of Theorem~\ref{cor:improp} by showing that its condition (II) is equivalent to \eqref{eq:ii}. (Cf.\ \cite[Lemma~1]{pitman56} and \cite[Lemma~11.2.1]{kawata}.) 

Indeed, $\PP(X_->x)\le x^{-\ell}\E X_-^\ell\ii{X_->x}$  
for all $x>0$, and so, \eqref{eq:ii} does imply condition (II). 

Vice versa, suppose now that condition (II) holds. Then for any $\vp\in(0,1-\frac p{\ell+1})$ and all real $x>1$ 
\begin{align*}
	\E X_-^{\ell+1}\ii{X_-\le x}
	&=\E\int_0^\infty(\ell+1)y^\ell\,\ii{y<X_-}\,\dd y\,\ii{X_-\le x} \\
	&\le\int_0^x(\ell+1)y^\ell\,\PP(X_->y)\,\dd y \\	&\le\int_0^{x^\vp}(\ell+1)y^\ell\,\dd y 
	+o\Big(\int_{x^\vp}^x(\ell+1)y^\ell\,\frac1{y^p}\dd y\Big)
	=o(x^{\ell+1-p})
\end{align*}
as $x\to\infty$. 
Also, for all $x>0$ 
\begin{align*}
	\E X_-^\ell\ii{X_->x}
	&=\E\int_0^\infty\ell y^{\ell-1}\,\ii{y<X_-}\,\dd y\,\ii{X_->x} \\
	&=\int_0^\infty\ell y^{\ell-1}\,\PP(X_->x\vee y)\,\dd y \\
	&=\int_0^x\ell y^{\ell-1}\,\dd y \;\PP(X_->x)\,
	+	\int_x^\infty\ell y^{\ell-1}\,\PP(X_->y)\,\dd y \\
	&=o\Big(\int_0^x\ell y^{\ell-1}\,\dd y \;\frac1{x^p}\,
	+	\int_x^\infty\ell y^{\ell-1}\,\frac1{y^p}\,\dd y\Big)
	=o(x^{\ell-p})
\end{align*}
as $x\to\infty$. 
So, condition (II) implies \eqref{eq:ii}. 
This completes Step~6 and thus the entire proof of Theorem~\ref{cor:improp}.
\end{proof}

%\vspace{.2cm} \textsc{Acknowledgment}\ \ .

%\begin{thebibliography}{}
%% Total bibitems: 43, successfull: 19, errors: 24
%% Generated through BatchMRef at http://www.e-publications.org

%\end{spacing}
%\end{proof}

\end{document}